\documentclass[11pt,letterpaper]{amsart}

\usepackage{amsfonts}
\usepackage{amssymb}
\usepackage{graphicx}
\usepackage{amsmath,mathtools}
\usepackage{latexsym}
\usepackage{amscd}
\usepackage{xypic}
\usepackage{mathrsfs}
\usepackage{enumitem} 
\usepackage{braket}
\usepackage{esint}
\usepackage{amsthm}
\usepackage{accents}
\usepackage{microtype}

\usepackage{hyperref}

\setlist{itemsep=3pt}

\allowdisplaybreaks

\newtheorem{prop}{Proposition}
\newtheorem{theo}[prop]{Theorem}
\newtheorem{lemm}[prop]{Lemma}
\newtheorem{coro}[prop]{Corollary}

\theoremstyle{definition}

\newtheorem{rema}[prop]{Remark}

\newtheorem{conj}{Conjecture}

\newcommand{\RR}{\mathbb{R}}
\renewcommand{\SS}{\mathbb{S}}

\newcommand{\cE}{\mathcal E}

\newcommand{\bangle}[1]{\left\langle #1 \right\rangle}

\let\oldmarginpar\marginpar
\renewcommand\marginpar[1]{\-\oldmarginpar[\raggedleft\footnotesize #1]%
{\raggedright\footnotesize #1}}

\DeclareMathOperator{\Id}{Id}

\DeclareMathOperator{\Div}{div}

\newcommand{\eps}{\varepsilon}

\usepackage{graphicx}

\allowdisplaybreaks

\title{On minimizers in the liquid drop model}

\author{Otis Chodosh}
\address{OC: Department of Mathematics, Bldg.\ 380, Stanford University, Stanford, CA 94305, USA}
\email{ochodosh@stanford.edu}

\author{Ian Ruohoniemi}
\address{IR: Department of Mathematics, University of Minnesota-Twin Cities, Minneapolis, MN,
55455, USA.}
\email{ruoho021@umn.edu}

\begin{document}
\maketitle

\begin{abstract}
We prove that round balls of volume $\leq 1$ uniquely minimize in Gamow's liquid drop model. 
\end{abstract}

\section{Introduction}  
For $\Omega\subset \RR^3$ measurable we define the liquid drop functional
\begin{equation}\label{eq:defn-E}
\cE(\Omega) : = P(\Omega) + D(\Omega). 
\end{equation}
where\begin{equation}\label{eq:defn-D}
D(\Omega) : = \frac 12 \iint_{\Omega\times\Omega} \frac{dxdy}{|x-y|}
\end{equation}
is the Coulomb energy and $P(\Omega) = \sup\left\{ \int_\Omega \Div X : X \in C^1_c(\RR^3;\RR^3), |X|\leq 1\right\}$ is the perimeter of $\Omega$. 

The functional \eqref{eq:defn-E} was introduced by Gamow in 1928 as a model for the nucleus of an atom \cite{Gamow}. The minimization problem
\begin{equation}\label{eq:profile}
E(V) : = \inf_{|\Omega| = V} \cE(\Omega). 
\end{equation}
has recently received considerable attention for its connections with various problems in mathematical physics. We refer to the recent survey \cite{CMT:Notices} for further discussion. 

One basic feature of \eqref{eq:defn-E} is the fact that the perimeter $P(\Omega)$ is minimized at a round ball (by the isoperimetric inequality) while the Coulomb term $D(\Omega)$ is maximized at a round ball (by Riesz rearrangement \cite{Riesz}). Thus, the two terms are in competition in the minimization problem \eqref{eq:profile}. A well-known conjecture suggests that the perimeter completely dominates in \eqref{eq:profile} until the Coulomb energy becomes strong enough to prevent the existence of minimizers.  

To state the conjecture precisely, write $B_1\subset \RR^3$ for the unit ball and set
\begin{equation}\label{eq:mass-to-split}
V_* : = \frac{2-2^{\frac 23}}{2^{\frac23}-1} \frac{|B_1| P(\partial B_1)}{\frac 12 \iint_{B_1\times B_1}|x-y|^{-1} dxdy} = 5\frac{2-2^{\frac 23}}{2^{\frac23}-1} \approx 3.51.
\end{equation}
The significance of \eqref{eq:mass-to-split} is that for $V > V_*$, the configuration of two balls of volume $V/2$ placed infinitely far apart has less energy than a single ball of volume $V$. 

We then have (cf.\ \cite{ChoksiPeletier}):
\begin{conj}
\label{conj:min-LD}\,
\begin{enumerate}[label=(\alph*)]
\item For $V\leq V_*$, the round ball of volume $V$ uniquely minimizes $\cE(\cdot)$ among measurable sets $\Omega\subset \RR^3$ with $|\Omega| = V$. 
\item For $V>V_*$ no minimizer exists.
\end{enumerate}
\end{conj}
Our main result confirms (a) for the range $V\leq 0.28V_*$:
\begin{theo}\label{theo:roundness}
Round balls $B_R$ with $|B_R| \leq 1$ uniquely minimize $\cE(\cdot)$ among measurable sets $\Omega\subset\RR^3$ with $|\Omega| = |B_R|$. 
\end{theo}
  Previously Kn\"upfer--Muratov proved \cite{KnupferMuratov1,KnupferMuratov2} (see also \cite{BonaciniCristoferi,Julin,MuratovZaleski,FigalliFuscoMaggiMillot,ChoksiNeumayerTopaloglu}) proved that balls uniquely minimize $\cE(\cdot)$ among sets of sufficiently small volume $V$. It seems that  Theorem \ref{theo:roundness} is the first result concerning (a) in Conjecture \ref{conj:min-LD} to go beyond the the ``isoperimetric dominated'' (small volume) regime where $D(\Omega)\ll P(\Omega)$ in \eqref{eq:defn-E} allows the repulsive term $D(\Omega)$ to be treated as a lower-order term (cf.\ \cite[p.\ 1132]{KnupferMuratov1}). Along these lines, it appears to be the first quantitative estimate (cf.\ \cite{MuratovZaleski}). 

\subsection{Boundary-interior Coulomb interaction}
A crucial ingredient in the proof of Theorem \ref{theo:roundness} is the following sharp inequality (see Theorem \ref{theo:est-for-R3-LD})
\begin{equation}\label{eq:P-bdry-D-main-ineq-intro}
|\Omega|^2 \leq\frac{1}{12\pi} P(\Omega) \iint_{\Omega\times \partial\Omega}\frac{dxdy}{|x-y|}.
\end{equation}
Equality holds here if and only if $\Omega$ is a round ball. The proof of \eqref{eq:P-bdry-D-main-ineq-intro} uses Santal\'o's formula from integral geometry. The significance of this inequality for us is that the integral appearing above arises in the first variation of $\cE(\cdot)$ in a natural way. 

Similar considerations (cf.\ \ref{eq:P-D-main-ineq}) also yield mean convexity of critical points of $\cE(\cdot)$ for a certain range of volumes (Proposition \ref{prop:mean-convex}) and an improved bound for the minimal binding energy per particle (Corollary \ref{coro:binding-est}). 
\subsection{Other results concerning Conjecture \ref{conj:min-LD}}
\begin{enumerate}
\item Kn\"upfer--Muratov \cite{KnupferMuratov2} as well as Lu--Otto \cite{LuOtto} proved that there is $V_{N} > 0$ large so that no minimizer exists for sets of volume $V > V_{N}$. Frank--Killip--Nam later obtained   \cite{FrankKillipNam} the quantitative estimate $V_{N} \leq 8 \approx 2.38 V_*$. 
\item Frank--Nam have proven \cite{FrankNam:existence-nonexistence} that a minimizer exists $\cE(\Omega)$ among sets with $|\Omega|=V$ for $V \leq V_*$ (this covers precisely the conjecturally sharp range of volumes). 
\end{enumerate}
See also \cite{AlbertiChoksiOtto,ChoksiPeletier1,ChoksiPeletier,RenWei:torus,CicaleseSpadaro,RenWei:2tori,KnupferMuratovNovaga:low-density,Frank:non-spherical-drops,EmmertFrankKonig}.

\subsection{Remarks}
\begin{enumerate}
\item Theorem \ref{theo:roundness} can be seen to  hold for volumes $\leq 1+\eps$, where $\eps>0$ is ineffective. This follows from the proof of Theorem \ref{theo:roundness} in Section \ref{sec:roundness}. 
\item The techniques used here could be extended to prove a result similar to Theorem \ref{theo:roundness} in higher dimensions $n\geq 4$ for the functional $P(\Omega) + \frac 12 \iint_{\Omega\times \Omega} \frac{dxdy}{|x-y|^{\alpha}}$, with $\alpha=n-2$. We hope to investigate this elsewhere. 
\item A sharp estimate of the form \eqref{eq:P-bdry-D-main-ineq-intro} for general $n$ and $\alpha$ does not seem to follow from our techniques, and thus our methods do not seem to establish an analogue of Theorem \ref{theo:roundness} for $n=2$ or $n\geq 3$ but $\alpha \not = n-2$.
\end{enumerate}

\subsection{Organization}
Section \ref{sec:newt-pot} recalls some facts about the Newtonian potential. Section \ref{sec:outer-min} contains a conditional outer-minimizing result for minimizers of $\cE(\cdot)$. In Section \ref{sec:santalo} we use integral geometry to prove new isoperimetric-type inequalities involving the Coulomb energy. We compute the first variation of energy in Section \ref{sec:first-var} and show how the isoperimetric-type inequalities, the Minkowski inequality, and the first variation combine in Section \ref{sec:Mink}. Finally, in Section \ref{sec:roundness} we prove Theorem \ref{theo:roundness} using a continuity argument and the results from Section \ref{sec:Mink}. Appendix \ref{proof:stantalo} contains a proof of Santal\'o's formula. Section \ref{sec:binding} contains a brief discussion of the maximal binding energy per particle.

\subsection{Acknowledgements}
O.C. was supported by Terman Fellowship and an NSF grant (DMS-2304432). We are grateful to Rupert Frank and Juncheng Wei for their interest and suggestions as well as the referees for useful comments concerning the exposition. 

\section{The Newtonian potential} \label{sec:newt-pot}
For $\Omega\subset \RR^3$ measurable, we recall that
\begin{equation}\label{eq:Newtonian-potential}
v_\Omega(x) = \int_\Omega \frac{dy}{|x-y|}
\end{equation}
is the \emph{Newtonian potential} of $\Omega$. 
\begin{lemm}\label{lemm:NewtonianBall}
The Newtonian potential of a round ball satisfies
\[
v_{B_R}(x) = \begin{cases}
\frac{2\pi}{3}(3R^2-|x|^2) & |x|\leq R\\
\frac{4\pi R^3}{3|x|} & |x|\geq R. 
\end{cases}
\]
In particular, the Coulomb energy satisfies $D(B_R) = \frac{16\pi^2}{15} R^5$. 
\end{lemm}
\begin{proof}
Note that $v_{B_R}(0) = 4\pi \int_0^R r dr = 2\pi R^2$. Moreover, $v_{B_R}$ is rotationally invariant and solves $\Delta v_{B_R} = -4\pi \chi_{B_R}$. Thus, $\frac{\partial^2 v_{B_R}}{\partial r^2} + \frac{2}{r} \frac{\partial v_{B_R}}{\partial r} = -4\pi$ for $r \in (0,R)$. This implies that $v_{B_R}(r) = \frac{2\pi}{3}(3R^2-r^2)$ for $r\leq R$. Arguing similarly for $r>R$ (and using $\lim_{r\to\infty}v_{B_R}(r) = 0$) we find $v_{B_R}(r) = \frac{4\pi R^3}{3r}$ for $r\geq R$. This proves the first assertion and the second follows from
\[
D(B_R) = 2 \pi\int_0^R \frac{2\pi}{3}(3R^2-r^2) r^2 dr = \frac{16\pi^2}{15} R^5.
\]
This completes the proof. 
\end{proof}

\begin{lemm}\label{lemm:riesz-Talenti}
For any measurable set $\Omega\subset\RR^3$ with $|\Omega| = |B_R|$, the  Newtonian potential satisfies $v_\Omega(x) \leq v_{B_R}(0) = 2\pi R^2$. 
\end{lemm}
\begin{proof}
Apply the ``bathtub principle'' \cite[Theorem 1.14]{LiebLoss} with $f(y) = -\frac{1}{|x-y|}$ and $G=|\Omega|$.
\end{proof}

\section{The outer-minimizing property}\label{sec:outer-min}
We say that a measurable Caccioppoli set $\Omega\subset \RR^3$ is \emph{outer-minimizing} if any $\Omega\subset \hat \Omega$ measurable has $P(\Omega) \leq P(\hat \Omega)$. We say that $\Omega\subset \RR^3$ measurable is \emph{strictly outer-minimizing} if any $\Omega\subset \hat \Omega$ measurable with $P(\hat\Omega) \leq P(\Omega)$ must have $|\hat\Omega\setminus\Omega|=0$. In this section we establish the following (conditional) outer-minimizing property for minimizers of $\cE$.

\begin{prop}\label{prop:outer-min}
Fix $V=\frac43\pi R^3 \leq 1$ and assume that $E(V) = \cE(B_R)$. If $\Omega\subset \RR^3$ has $|\Omega|=V$ and $\cE(\Omega) = E(V)$ then $\Omega$ is strictly outer-minimizing.
\end{prop}
\begin{proof}
Consider $\Omega\subset \hat \Omega \subset \RR^3$ measurable with $P(\hat\Omega) \leq P(\Omega)$. Write $|\hat\Omega| = \frac 43\pi R^3 (1+\alpha^3)$ for some $\alpha \geq 0$. Our goal is to prove that $\alpha=0$. 

Write $\Gamma : = \hat\Omega\setminus \Omega$ and $\tilde \Omega = (1+\alpha^3)^{-\frac 13} \hat\Omega$. Note that $|\tilde\Omega| = |\Omega|$ and $|\Gamma| = \frac{4}{3}\pi R^3\alpha^3$. The assumption $P(\hat\Omega) \leq P(\Omega)$ implies
\begin{align*}
P(\tilde\Omega) & = (1+\alpha^3)^{-\frac 23}P(\hat\Omega) \leq (1+\alpha^3)^{-\frac23}P(\Omega) .
\end{align*}
Recalling the definition of the Newtonian potential $v_\Omega$ in \eqref{eq:Newtonian-potential} we have
\begin{align*}
D(\tilde\Omega) & = (1+\alpha^3)^{-\frac53}D(\hat\Omega) \\
& = (1+\alpha^3)^{-\frac53}D(\Omega) + (1+\alpha^3)^{-\frac53} \int_\Gamma v_\Omega(x) dx + (1+\alpha^3)^{-\frac53} D(\Gamma) \\
& \leq (1+\alpha^3)^{-\frac23}D(\Omega) +  (1+\alpha^3)^{-\frac53} 2\pi R^2 |\Gamma| + (1+\alpha^3)^{-\frac53} D(B_{\alpha R}) \\
& = (1+\alpha^3)^{-\frac23}D(\Omega) +  (1+\alpha^3)^{-\frac53} \frac{8\pi^2}{3} R^5\alpha^3 + (1+\alpha^3)^{-\frac53} \frac{16\pi^2}{15} R^5\alpha^5.
\end{align*}
In the third line we used the estimate $(1+\alpha^3)^{-\frac53}\leq (1+\alpha^3)^{-\frac23}$, Lemma \ref{lemm:riesz-Talenti}, and the fact that balls maximize $D(\cdot)$ for fixed volumes. Since $\cE(\Omega)= E(V)\leq \cE(\tilde \Omega)$ we find
\begin{align*}
\cE(\Omega) \leq (1+\alpha^3)^{-\frac23} \cE(\Omega)  +  (1+\alpha^3)^{-\frac53} \frac{8\pi^2}{3} R^5\alpha^3 + (1+\alpha^3)^{-\frac53} \frac{16\pi^2}{15} R^5\alpha^5. 
\end{align*}
Rearranging and using $\cE(\Omega) = E(V)=\cE(B_R)$ we find 
\begin{align*}
0 \leq \left (1+\alpha^3 - (1+\alpha^3)^{\frac53}\right) \cE(B_R)  +  \frac{8\pi^2}{3} R^5\alpha^3 +  \frac{16\pi^2}{15} R^5\alpha^5. 
\end{align*}
Using $\cE(B_R) = 4\pi R^2 + \frac{16\pi^2}{15}R^5$ this can be simplified to read
\begin{equation}\label{eq:outer-min-alpha-eqn-genR}
0 \leq \left (1+\alpha^3 - (1+\alpha^3)^{\frac53}\right) \left( \frac {15}{\pi} R^{-3} + 4 \right)  + 10 \alpha^3 + 4 \alpha^5. 
\end{equation} 
Since $\frac{15}{\pi} R^{-3} \geq 20$ and $1+\alpha^3 - (1+\alpha^3)^{\frac53} < 0$ we find
\begin{equation}\label{eq:ineq-for-alpha}
0 \leq   12\left(1+\alpha^3 - (1+\alpha^3)^{\frac53}\right)  + 5 \alpha^3 + 2 \alpha^5. 
\end{equation} 
We now use \eqref{eq:ineq-for-alpha} to prove that $\alpha=0$. Since $(1+\alpha^3)^{\frac 53} \geq 1+\frac 53 \alpha^3$, we obtain
\[
0 \leq  -8\alpha^3  + 5 \alpha^3 + 2 \alpha^5
\]
which gives 
\begin{equation}\label{eq:imp-alpha}
\alpha>0 \quad \Rightarrow \quad \alpha^2 \geq \frac 32.
\end{equation}
 On the other hand we have that
\[
(1+\alpha^3)^{\frac53} = \alpha^5(1+\alpha^{-3})^{\frac53}\geq \alpha^5 + \frac 53 \alpha^2.
\]
Using this instead in \eqref{eq:ineq-for-alpha} we find
\[
0 \leq   12 - 20 \alpha^2 + 17 \alpha^3  - 10 \alpha^5 : = f(\alpha). 
\]
We note that $f(1) = -1$. Moreover,
\[
f'(\alpha) = -\alpha(40-51\alpha+50\alpha^3) : = -\alpha g(\alpha). 
\]
It's easy to check that $g(1) = 39$. Since
\[
g'(\alpha) = 150\alpha^2-51 \geq 0
\]
for $\alpha \geq 1$ we find that $g(\alpha) >0$ for $\alpha\geq 1$. Thus $f'(\alpha) < 0$ for $\alpha \geq 1$. Finally, this implies that $f(\alpha) <0$ for $\alpha\geq 1$, so \eqref{eq:ineq-for-alpha} implies that $\alpha <1$. Combined with \eqref{eq:imp-alpha} this gives $\alpha = 0$ as claimed.
\end{proof}

\begin{rema}
The above proof could be extended (with some modifications) to cover the range $V\leq \frac{20}{11}$. 
\end{rema}

\section{Santal\'o's formula and isoperimetric-type estimates} \label{sec:santalo}
One of the ingredient in the proof of Theorem \ref{theo:roundness} are new isoperimetric-type inequalities relating Coulomb energy terms to perimeter and volume. To state these inequalities, we recall \eqref{eq:defn-D} and also define\footnote{It's important to note the factor of $\frac 12$ that appears in \eqref{eq:defn-D} does not appear in \eqref{eq:D-partial}.}
\begin{equation}\label{eq:D-partial}
D^\partial(\Omega) = \iint_{\Omega\times \partial\Omega} \frac{dxdy}{|x-y|} = \int_{\partial\Omega}v_\Omega(x) dx, 
\end{equation} 
where the Newtonian potential $v_\Omega(x)$ is defined in \eqref{eq:Newtonian-potential}. In this section we prove:
\begin{theo}\label{theo:est-for-R3-LD}
For a non-empty compact set $\Omega\subset \RR^3$ with smooth boundary we have 
\begin{equation}\label{eq:P-D-main-ineq}
|\Omega|^3 < \frac{3}{16\pi} P(\Omega)^2  D(\Omega)
\end{equation}
and
\begin{equation}\label{eq:P-bdry-D-main-ineq}
|\Omega|^2 \leq\frac{1}{12\pi} P(\Omega) D^\partial(\Omega)
\end{equation}
Equality holds in \eqref{eq:P-bdry-D-main-ineq} if and only if $\Omega$ is a round ball. 
\end{theo}
The proof of Theorem \ref{theo:est-for-R3-LD} relies on an integro-geometric formula due to Santal\'o \cite{Santalo}. We refer to e.g., \cite{BanchoffPohl,Croke:iso,Croke:4diso,Croke:synthetic,KloKup,HoisingtonMcGrath} for related results. 
\begin{rema}
Kn\"upfer--Muratov proved that \eqref{eq:P-D-main-ineq} holds with some abstract constant via an interpolation argument \cite[Lemma 7.1]{KnupferMuratov2} and more recently, Kn\"upfer--Muratov--Novaga \cite[Theorem 2.2]{KnupferMuratovNovaga} proved\footnote{The inequality is not stated explicitly in this form in \cite{KnupferMuratovNovaga} but follows from their work in a straightforward manner; cf.\ Lemma \ref{lem:estar-ineq}.} a version of \eqref{eq:P-D-main-ineq} with an explicit constant (they obtain $\frac{4}{9\pi}$ in place of $\frac{3}{16\pi}$). Our proof of \eqref{eq:P-D-main-ineq} is completely different from theirs and yields a better (but still not sharp) constant. See also Appendix \ref{sec:binding} some further discussion. 
\end{rema}

\subsection{Santal\'o's formula} Suppose that $\Omega\subset \RR^N$ is a compact set with smooth boundary. Let $U : = \Omega\times \SS^{N-1} \subset T\Omega$ denote the unit sphere bundle over $\Omega$. We write $U_x:=\{x\}\times \SS^{N-1}$ for the fiber over $x$. For $\nu(x)$ the \emph{inwards} pointing unit normal to $\partial\Omega$, we set
\[
U^+ : = \{(x,\sigma) \in \partial\Omega \times \SS^{N-1} :  \bangle{\sigma,\nu(x)} > 0\}
\]
and $U_x^+ : = U^+\cap U_x$. Below we often write $u =(x,\sigma) \in U^+$ and $\cos u = \bangle{\sigma,\nu(x)} $. Observe that $U$ (resp.\ $U^+$) is a smooth submanifold of $\RR^N \times \RR^N$ and thus inherits a Riemannian metric. We will write $du$ for the induced volume measure on both manifolds.  We define a function $\ell : U \to [0,\infty)$ defined by
\[
\ell(x,\sigma) : = \sup\{t \geq 0 : x + s \sigma \in \Omega \textrm{ for all } 0 \leq s \leq t\}.
\]
The following result is the main tool used in the proof of Theorem \ref{theo:est-for-R3-LD}. 
\begin{theo}[Santal\'o]\label{theo:santalo}
Set $\Phi ((x,\sigma),t) = (x+t\sigma,\sigma)$ for $t < \ell(x,\sigma)$.  Then,
\[
\int_{U} f(u) du = \int_{U^{+}} \int_{0}^{\ell(u)} f(\Phi(u,t)) \cos u \, dt du
\]
for any $f \in L^1(U)$. 
\end{theo}
This is proven in \cite[pp.\ 336--338]{Santalo} (see also \cite[p.\ 286]{Berger:geo}) (in the more general setting of the unit sphere bundle on a Riemannian manifold). We sketch the proof of Theorem \ref{theo:santalo} in Appendix \ref{proof:stantalo}. In the sequel will only ever use the following special case. 
\begin{prop} \label{prop:santalo-ell}
For $\alpha > -1$ we have
\[
\int_U \ell(u)^\alpha du = \frac{1}{\alpha+1} \int_{U^+} \ell(u)^{\alpha+1} \cos u\, du . 
\]
\end{prop}
\begin{proof}
We take $f(u) = \ell(u)^\alpha$ in Theorem \ref{theo:santalo}. Note that for $u \in U^+$ and $t \in (0,\ell(u))$, $\ell(\Phi(u,t)) = \ell(u) - t$.  Thus
\[
\int_U \ell(u)^\alpha du = \int_{U^+} \int_0^{\ell(u)} (\ell(u)-t)^\alpha \, \cos u\, dt du = \frac{1}{\alpha+1} \int_{U^+} \ell(u)^{\alpha+1} \cos u\, du .
\]
This completes the proof. 
\end{proof}
Taking $\alpha=0$ yields the following commonly used expression:
\begin{coro}\label{coro:volume-santalo}
We have
\[
|\SS^{N-1}| |\Omega| = \int_{U^+} \ell(u) \cos u \, du. 
\]
\end{coro}

\subsection{Integral estimates for Coulomb energies} We now show how Santal\'o's formula can be used to prove Theorem \ref{theo:est-for-R3-LD}. 

\begin{lemm}\label{lemm:D-est}
For $\Omega\subset \RR^3$ compact with smooth boundary we have
\[
\int_{U^+} \ell(u)^3 \cos u \, du \leq 12 D(\Omega). 
\]
\end{lemm}
\begin{proof}
For $x \in\Omega$, the volume form $dy$ in spherical coordinates centered at $x$ $(r,\sigma) \in (0,\infty)\times U_x$ is $r^2 dr d\sigma$ (where $d\sigma$ is the volume form on the unit sphere $U_x$). This gives
\[
\int_\Omega \frac{dy}{|x-y|} \geq \int_{U_{x}} \int_{0}^{\ell(x,\sigma)} r^{-1} r^{2} dr d\sigma = \frac 12 \int_{U_x} \ell(x,\sigma)^2 d\sigma .
\]
Integrating with respect to $x$, we obtain 
\begin{align*}
D(\Omega) & \geq \frac 14  \int_{ \Omega} \int_{U_{x}}\ell(x,\sigma)^{2} \, d\sigma dx\\
& = \frac 14 \int_{U} \ell(u)^{2} \, du\\
& = \frac {1}{12} \int_{U^{+}} \ell(u)^{3} \cos u \, du.
\end{align*}
We used Santal\'o's formula (in the form of Proposition \ref{prop:santalo-ell}) in the final equality above. 
\end{proof}
Using this we can now prove the first inequality in Theorem \ref{theo:est-for-R3-LD}. 
\begin{proof}[Proof of \eqref{eq:P-D-main-ineq}]
Corollary \ref{coro:volume-santalo} gives
\[
4\pi |\Omega| = \int_{U^+} \ell(u) \cos u \, du.
\]
Thus, using H\"older's inequality we find
\begin{align*}
(4\pi)^3 |\Omega|^3 & \leq  \left(\int_{U^+} \cos u \, du \right)^2\left( \int_{U^+} \ell(u)^3 \cos  u \, du\right) \\
& = P(\Omega)^2 \left( \int_{\SS^2_+} \bangle{e_3,\sigma} \, d\sigma \right)^2 \int_{U^+} \ell(u)^3 \cos  u \, du \\
& = \pi^2 P(\Omega)^2 \int_{U^+} \ell(u)^3 \cos  u \, du  \leq 12\pi^2 P(\Omega)^2 D(\Omega)
\end{align*}
where we used Lemma \ref{lemm:D-est} in the final step and wrote $\SS^2_+ = \{\sigma \in \SS^2 : \bangle{e_3,\sigma} > 0\}$. To complete the proof, we note that if equality held in H\"older's inequality then $\ell(u)=K$ for a.e.\ $u\in U^+$ and $K$ some fixed constant. This cannot occur, so the inequality at that step must be strict. This completes the proof. 
\end{proof}

\begin{lemm}
For $\Omega\subset \RR^3$ compact with smooth boundary 
\[
\int_{U^+} \ell(u)^2 \, du \leq 2 D^\partial(\Omega) .
\]
\end{lemm}
\begin{proof}
For $x \in \partial\Omega$, we can use spherical coordinates centered at $x$ to write
\[
\int_{\Omega} \frac{dy}{|x-y|} \geq \int_{U_x^+} \int_0^{\ell(x,\sigma)} r^{-1} r^2 drd\sigma = \frac 12 \int_{U_x^+} \ell(x,\sigma)^2 d\sigma.
\]
Integrating with respect to $x\in\partial\Omega$ completes the proof.
\end{proof}
Using this we can now prove the second inequality in Theorem \ref{theo:est-for-R3-LD}. 
\begin{proof}[Proof of \eqref{eq:P-bdry-D-main-ineq}]
Corollary \ref{coro:volume-santalo} gives
\[
4\pi |\Omega| = \int_{U^+} \ell(u) \cos  u \, du.
\]
Thus, using H\"older's inequality we find
\begin{align*}
16\pi^2 |\Omega|^2 & \leq \left( \int_{U^+} \cos^2 u \, du \right) \left( \int_{U^+} \ell(u)^2 du \right) \\
& = P(\Omega) \left( \int_{\SS^2_+} \bangle{e_3,\sigma}^2 \, d\sigma \right) \left( \int_{U^+} \ell(u)^2 du \right) \\
& = \frac{2\pi}{3} P(\Omega)  \left( \int_{U^+} \ell(u)^2 du \right)  = \frac{4\pi}{3} P(\Omega) D^\partial(\Omega) 
\end{align*}
Rearranging this proves the inequality. It is easy to see that equality holds when $\Omega$ is a ball. Conversely, if equality holds in the application of H\"older's inequality we find that $\ell(u) = K \cos u$ for a.e.\ $u\in U^+$ so $\Omega$ must be a round ball. This completes the proof. 
\end{proof}

\section{The first variation of energy}\label{sec:first-var}

In this section we compute the first variation of $\cE(\cdot)$ and particular deduce the value of the Lagrange multiplier at a volume constrained critical point following the scaling method of \cite{BarbosaDoCarmo}. Such a calculation has appeared previously in several places in the literature including \cite[Theorem 3.3]{BonaciniCristoferi}, \cite[Lemma 12]{Frank:non-spherical-drops}, and \cite[Lemma 2]{Julin:analytic}. 

In this paper we write $H_{\partial\Omega}$ for the sum of the principal curvatures (not the average). 

\begin{lemm}\label{lemm:first-var}
If $\Omega\subset \RR^3$ is a compact set with smooth boundary that's stationary with respect to  $\cE(\cdot)$ for volume preserving variations then
\[
H_{\partial\Omega} + v_{\Omega} = \frac{2P(\Omega)}{3|\Omega|}  +   \frac{5D(\Omega)}{3|\Omega|} 
\]
along $\partial\Omega$. 
\end{lemm}
\begin{proof}
The first variation of area implies that $H_{\partial\Omega} + v_{\Omega}=\lambda$ is constant along $\partial\Omega$. It thus remains to determine $\lambda$. For $t \in (-\delta,\delta)$ (for some $\delta>0$ small) we denote  $\tilde\Omega_t$ the family of regions obtained by flowing $\partial\Omega$ with unit normal speed (to the outside). Then, we choose $\eta_t$ so that $t\mapsto \eta_t \tilde\Omega_t$ has constant volume. Because $\frac{d}{dt}\big|_{t=0}|\tilde\Omega_t|=P(\Omega)$ we find 
\[
\frac{d}{dt}\Big|_{t=0} \eta_t  = - \frac{P(\Omega)}{3 |\Omega|}
\]
(using the first variation of volume). Thus
\[
\cE(\eta_t \tilde\Omega_t) = \eta_t^2 P(\tilde \Omega_t) + \eta_t^5 \frac 12 \iint_{\tilde\Omega_t\times \tilde\Omega_t} \frac{dxdy}{|x-y|} .
\]
Differentiating this at $t=0$ and using the first variation formulae from the proof of  \cite[Theorem 3.3]{BonaciniCristoferi} we find 
\[
0 = - \frac{2P(\Omega)^2}{3 |\Omega|} - \frac{5 P(\Omega)}{6|\Omega|} \iint_{\Omega\times \Omega} \frac{dxdy}{|x-y|} + \int_{\partial\Omega} (H_{\partial\Omega}(x)  + v_\Omega(x))dx. 
\]
Using $H_{\partial\Omega}(x)  + v_\Omega(x)=\lambda$ is constant, we have proven the assertion.
\end{proof}

The following result is not used in the proof of Theorem \ref{theo:roundness} but may be of some independent interest.

\begin{prop}\label{prop:mean-convex}
Suppose that $\Omega\subset \RR^3$ is a compact set with smooth boundary stationary with respect to $\cE(\cdot)$ for variations fixing $|\Omega| = V \leq \frac 4 3 \left(\frac{10}{3}\right)^{\frac 12} \approx 2.43$. Then $\partial\Omega$ is connected and has positive mean curvature $H_{\partial\Omega} > 0$. 
\end{prop}
\begin{proof}
Write $|\Omega| = V = |B_R|$. Combining Lemmas \ref{lemm:riesz-Talenti} and \ref{lemm:first-var} with \eqref{eq:P-D-main-ineq} we find
\begin{align*}
H_{\partial\Omega} + 2\pi R^2 
& >  \frac{2P(\Omega)}{3|\Omega|}  +  \frac{80 \pi|\Omega|^2}{9P(\Omega)^2} \geq 2 \left( \frac{10\pi}{3}\right)^{\frac 13},
\end{align*}
where we used AM-GM in the last step. Thus, $H_{\partial\Omega}>0$ for $R \leq \left( \frac{10}{3\pi^2}\right)^{\frac 16}$. This completes the proof of mean-convexity.

Finally, if $\partial\Omega$ has more than one component, then there will be some component $\Sigma$ so that $\Sigma = \partial\Gamma$ for a bounded open component of $\RR^3\setminus\Omega$. By touching $\Sigma$ from the outside by a sphere, we see that the mean curvature of $\Sigma$ is positive at some $p\in\Sigma$, when measured with respect to $\Gamma$.  However, this means that the mean curvature at $p$ when measured with respect to $\Omega$ is negative. This contradicts the above estimate.
\end{proof}

\section{Roundness via the Minkowski inequality} \label{sec:Mink}
Suppose that $\Omega\subset \RR^3$ is a compact set with smooth boundary. We say that $\Omega$ satisfies the \emph{Minkowski inequality} if 
\begin{equation}\label{eq:Mink}
(16\pi P(\Omega))^{\frac 12} \leq \int_{\partial\Omega} H_{\partial\Omega}(x)dx . 
\end{equation}
Minkowski proved \cite{Minkowski} that if $\Omega$ is convex then \eqref{eq:Mink} holds. We will need the following result of Huisken (see \cite[Corollary 1.3]{AgostinianiFogagnoloMazzieri} for a proof):
\begin{theo}\label{theo:huisken-mink}
If $\Omega\subset\RR^3$ is a compact outer-minimizing set with smooth boundary then $\Omega$ satisfies the Minkowski inequality \eqref{eq:Mink}. If equality holds then $\Omega$ is a round ball. 
\end{theo}
It's conjectured that ``outer-minimizing'' can be relaxed to ``mean-convex.'' See \cite{GuanLi,DalphinHenrotMasnouTakahashi,Glaudo,CEK,Brendle} for related results.

\begin{prop}\label{prop:mink-implies-roundness-min}
Suppose that a compact set $\Omega\subset \RR^3$ has smooth boundary and minimizes $\cE(\cdot)$ among measurable regions with fixed volume $|\Omega| = V \leq 1$. If $\Omega$ satisfies the Minkowski inequality \eqref{eq:Mink} then $\Omega$ agrees with a round ball. 
\end{prop}
\begin{proof}
Write $|\Omega| = V = \frac 4 3 \pi R^3$. Lemma \ref{lemm:first-var} yields 
\[
\int_{\partial\Omega} H_{\partial\Omega}(x) dx + D^\partial(\Omega) = \frac{2P(\Omega)^2}{3|\Omega|} + \frac{5 P(\Omega)}{3|\Omega|} D(\Omega)
\]
Combining this with the Minkowski inequality \eqref{eq:Mink} (assumed to hold) and the isoperimetric-type inequality from \eqref{eq:P-bdry-D-main-ineq}, we find 
\[
(16\pi P(\Omega))^{\frac 12} + \frac{12\pi |\Omega|^2}{P(\Omega)} \leq \frac{2P(\Omega)^2}{3|\Omega|} + \frac{5 P(\Omega)}{3|\Omega|} D(\Omega)
\]
Because $\Omega$ minimizes $\cE(\cdot)$ among sets of the same volume, we have $D(\Omega) \leq \cE(B_R) - P(\Omega)$. Thus, we arrive at the expression
\begin{equation}\label{eq:roundness-geo-expression}
(16\pi P(\Omega))^{\frac 12} + \frac{12\pi |\Omega|^2}{P(\Omega)} \leq  \frac{5 P(\Omega)}{3|\Omega|} \cE(B_R) - \frac{ P(\Omega)^2}{|\Omega|}. 
\end{equation}
Note that $x : = P(\Omega) \geq 4\pi R^2$ by the isoperimetric inequality with equality only for the round ball. We will show that \eqref{eq:roundness-geo-expression} implies that $x = 4\pi R^2$. Using
\[
\cE(B_R) = 4\pi R^2 + \frac{16\pi^2}{15}R^5
\]
the inequality \eqref{eq:roundness-geo-expression} can be written as
\[
f_R(x) := - \frac{3}{4\pi} R^{-3} x^2 + \left( 5R^{-1} + \frac 4 3 \pi R^2 \right) x  - 4 \pi^{\frac 12} x^{\frac 12} - \frac{64\pi^3}{3} R^6 x^{-1} \geq 0.
\]
We will check that $f_R(4\pi R^2) = 0$, $f'(4\pi R^2) \leq 0$ (as long as $V \leq 1$), and $f''_R(x) <0$ for $x \geq 4\pi R^2$. From this, we see that $x = 4\pi R^2$ as desired. 

Since all inequalities used above hold as equalities at the round ball, we deduce that $f_R(4\pi R^2) = 0$. Next, we compute
\[
f'_R(4\pi R^2) = - 2 R^{-1} + \frac{8\pi}{3} R^2 
\]
so $f'_R(4\pi R^2)\leq 0$ for $V = \frac43\pi R^3 \leq 1$. Finally, we compute
\begin{align*}
f''_R(x) & = - \frac{3}{2\pi} R^{-3}+ \pi^{\frac 12} x^{-\frac 32}  - \frac{128 \pi^3}{3} R^6 x^{-3} .
\end{align*}
Using $x^{-\frac 32} \leq (4\pi R^2)^{-\frac 32}$, we can estimate
\[
- \frac{3}{2\pi} R^{-3}+ \pi^{\frac 12} x^{-\frac 32} \leq - \frac{11}{8\pi} R^{-3},
\]
implying that $f''_R(x) < 0$ for $x\geq 4\pi R^2$. This completes the proof.
\end{proof}

\section{Proof of Theorem \ref{theo:roundness}} \label{sec:roundness}

Recall the definition of $E(V)$ in \eqref{eq:profile}. The map $V\mapsto E(V)$ is continuous (cf.\ \cite[p.\ 4444]{FrankLieb:compact-least-density}). Let $E_B(V) : = \cE(B_{(\frac{3V}{4\pi})^{\frac 13}})$ denote the profile corresponding to the energy of the round ball of volume $V$ (clearly $E_B(V)$ is also continuous). Let 
\begin{equation}\label{eq:definVB}
V_B : = \sup\{V >0 : E(V') = E_B(V') \textrm{ for all $V'\leq V$}\}.
\end{equation}
By \cite[Theorem 3.2]{KnupferMuratov2} (cf.\ \cite{BonaciniCristoferi,Julin,FigalliFuscoMaggiMillot}) we have that $V_B>0$. Moreover, by \cite[Theorem 2.10]{BonaciniCristoferi} if $V<V_B$ then the ball uniquely minimizes $\cE(\cdot)$ for fixed volume $V$ (and minimizes for $V=V_B$).

\begin{lemm}\label{lemm:cont-method}
If $V_B<V_*$ there is $\Omega\subset \RR^3$ compact with smooth boundary so that $|\Omega|=V_B$, $\cE(\Omega) = E(V_B)$ but $\Omega$ is not a round ball. 
\end{lemm}
\begin{proof}
By  \cite[Theorem 1]{FrankNam:existence-nonexistence} and the definition of $V_B$ there are $\Omega_j\subset \RR^3$ that minimize $\cE(\Omega_j)$ among sets of fixed volume so that $|\Omega_j| \searrow V_B$ but $\Omega_j$ are not round balls. By \cite[Lemma 7.2]{KnupferMuratov2} we may assume that the sets $\Omega_j$ are uniformly essentially bounded. Thus, by standard compactness results for sets of finite perimeter (cf.\ \cite[\S 2.6]{Simon:GMT}) we can pass to a subsequence with $\Omega_j \to \Omega$ in $L^1$ and $P(\Omega) \leq \limsup_{j\to\infty} P(\Omega_j)$. As in \cite[Lemma 2.3]{FrankLieb:compact-least-density} we have $D(\Omega_j)\to D(\Omega)$. Thus, by continuity of $E(\cdot)$ we find
\[
\cE(\Omega) \leq E(V_B),
\]
so $\Omega$ is a minimizer of $\cE(\cdot)$ among sets of volume $V_B$. By \cite[Theorem 2.7]{BonaciniCristoferi}, we can assume that $\Omega$ is compact with  smooth boundary. 

Suppose that $\Omega$ was a round ball (up to a set of measure zero). Then, the strict stability of the ball (in the range of volumes under consideration) implies that $\Omega_j$ must also be a round ball for $j$ sufficiently large (see the proof of \cite[Theorem 2.10]{BonaciniCristoferi}). This contradicts our assumption that the $\Omega_j$ were not round balls. This completes the proof. 
\end{proof}

With these preparations we can now finish the proof of our main theorem. Note that the proof below actually proves that there is $\eps>0$ so that round balls of volume $\leq 1+\eps$ uniquely minimize $\cE(\cdot)$ among competing regions of fixed volume. 
\begin{proof}[Proof of Theorem \ref{theo:roundness}]
As observed above, it suffices to prove that $V_B$ (defined in \eqref{eq:definVB}) satisfies $V_B > 1$. For the sake of contradiction, assume that $V_B \leq 1 < V_*$. By Lemma \ref{lemm:cont-method} there exists $\Omega\subset \RR^3$ compact with smooth boundary so that $|\Omega|=V_B$, $\cE(\Omega) = E(V_B)$, but $\Omega$ is not a round ball. 

Let $R= \left(\frac{3V_B}{4\pi}\right)^{\frac 13}$ so that $|\Omega| = |B_R|$ and $\cE(\Omega)=\cE(B_R)$. Because we are assuming that $V_B\leq 1$, the condition $\cE(\Omega)=\cE(B_R)=E(V_B)$ allows us to apply Proposition \ref{prop:outer-min} to conclude that $\Omega$ is outer-minimizing. Thus, by Huisken's Theorem \ref{theo:huisken-mink}, $\Omega$ satisfies the Minkowski inequality \eqref{eq:Mink}. Since $V_B\leq1$ we can then apply Proposition \ref{prop:mink-implies-roundness-min} to conclude that $\Omega$ is a round ball. This is a contradiction.
\end{proof}

\appendix

\section{Proof of Theorem \ref{theo:santalo}}\label{proof:stantalo}

Write $U^\circ$ for the interior of $U$ and
\[
\tilde U : = \{(u,t) \in U^+ \times (0,\infty) : t < \ell(u)\}. 
\]
Then, we define 
\[
\Phi : \tilde U \to U^\circ, \qquad \Phi : ((x,\sigma),t) \mapsto (x+t\sigma,\sigma).
\]
The map $\Phi$ is a diffeomorphism onto its image and $U^\circ\setminus \Phi(\tilde U)$ has measure zero. We now compute the Jacobian of $\Phi$ at $p : = ((x,\sigma),t) \in T\tilde U$. We can assume that $T_x\partial\Omega = \RR^{N-1}\times \{0\}$. Note that (using the splitting $T_p\tilde U = T_x\partial\Omega \times T_\sigma \SS^{N-1}\times \RR$ and $T_{\Phi(p)}U = T_{x+t\sigma}\RR^N \times T_\sigma \SS^{n-1}$):
\begin{align*}
d\Phi_p(v,0,0) & = (v,0)\\
d\Phi_p(0,w,0) & = (tw,w)\\
d\Phi_p(0,0,1) & = (\sigma,0).
\end{align*}
In particular, if we choose an orthonormal basis $w_1,\dots,w_{N-1}$ for $T_\sigma \SS^{N -1}$ and use $e_1,\dots,e_N$ for the standard basis on $\RR^N$ we find that the matrix of $d\Phi_p$ becomes 
\[
\left( \begin{matrix} \Id_{(N-1)\times (N-1)} & * & *   \\
0_{1\times (n-1)} & t & \bangle{\sigma,e_N} \\
0_{(N-1) \times (N-1)} & \Id_{(N-1)\times (N-1)} & 0_{(N-1)\times 1} \end{matrix} \right)
\]
which has determinant $\pm \bangle{\sigma,e_N} = \pm \cos u$. The assertion thus follows from the change of variables formula.

\section{Binding energy per particle} \label{sec:binding} 
We note that the sharp constant in \eqref{eq:P-D-main-ineq} is closely related to a question posed by Frank--Lieb \cite{FrankLieb:compact-least-density} concerning the quantity
\begin{equation}\label{eq:min-binding}
e_* : = \min_{\Omega\subset \RR^3} \frac{\cE(\Omega)}{|\Omega|} 
\end{equation}
where the minimum is taken over all measurable sets with $0<|\Omega|<\infty$. By \cite[Remark 3.6]{FrankLieb:compact-least-density} the sharp constant in \eqref{eq:P-D-main-ineq} is closely related to $e_*$:
\begin{lemm}\label{lem:estar-ineq}
For $\Omega\subset \RR^3$ measurable we have
\[
|\Omega|^3 \leq\frac{27}{4e_*^3} P(\Omega)^2 D(\Omega) .
\]
Equality holds if and only if $\cE(\Omega) = e_*|\Omega|$. 
\end{lemm}

As such, the inequality \eqref{eq:P-D-main-ineq} yields the following estimate.
\begin{coro}\label{coro:binding-est}
$e_* >  (36\pi)^{\frac 13} \approx 3.836$. 
\end{coro}

\begin{rema}
Note that in \cite[Theorem 2.2]{KnupferMuratovNovaga} Kn\"upfer--Muratov--Novaga  derived the lower bound $e_*  \geq\left(\frac{243\pi}{16}\right)^{\frac 13} \approx 3.627$.
\end{rema}
One can compare these estimates to the upper bound
\[
e_* \leq \min_R \frac{\cE(B_R)}{|B_R|} = 3 \left( \frac{9 \pi}{5} \right)^{\frac 13} \approx 5.345
\] 
obtained by optimizing the energy density of round balls.

\bibliographystyle{alpha}
\bibliography{bib}

\end{document}